\title{On a problem of Duke-Erd\H{o}s-R\"odl on
cycle-connected subgraphs}

\author{
Jacob Fox\thanks{Department of Mathematics, Princeton, Princeton, NJ.
Email: {\tt
jacobfox@math.princeton.edu}. Research supported by an NSF Graduate
Research Fellowship and a Princeton Centennial Fellowship.} \and
Benny Sudakov\thanks{ Department of Mathematics,
UCLA,  Los Angeles, CA 90095.
Email: {\tt bsudakov@math.ucla.edu}. Research
supported in part by NSF CAREER award DMS-0546523, NSF grant
DMS-0355497 and by a USA-Israeli BSF grant.}}

\documentclass[11pt]{article}
\usepackage{amsfonts,amssymb,amsmath,latexsym}

\newenvironment{proof}
      {\medskip\noindent{\bf Proof:}\hspace{1mm}}
      {\hfill$\Box$\medskip}

\def\qed{\ifvmode\mbox{ }\else\unskip\fi\hskip 1em plus 10fill$\Box$}

\newtheorem{theorem}{Theorem}[section]

\newtheorem{lemma}[theorem]{Lemma}

\newtheorem{problem}[theorem]{Problem}

\begin{document}
\date{}

\maketitle
\begin{abstract}
In this short note, we prove that for $\beta <1/5$ every graph $G$ with $n$
vertices and $n^{2-\beta}$ edges contains a subgraph $G'$
with at least $c n^{2-2\beta}$ edges such that every pair of
edges in $G'$ lie together on a cycle of length at most 8.
Moreover edges in $G'$ which share a vertex lie together on a cycle of length at most 6.
This result is best possible up to the constant factor and settles
a conjecture of Duke, Erd\H{o}s, and R\"odl.
\end{abstract}

\section{Introduction}
Let $\mathcal{H}$ be a fixed collection of graphs. A graph $G$ is
{\it $\mathcal{H}$-connected} if every pair of edges of $G$ is
contained in a subgraph $H$ of $G$, where $H$ is a member of
$\mathcal{H}$. For example, if $\mathcal{H}$ is the collection of
all paths, then, ignoring isolated vertices,
$\mathcal{H}$-connectedness is equivalent to connectedness. If
$\mathcal{H}$ consists of all paths of length at most $d$, then each
$\mathcal{H}$-connected graph has a diameter at most $d$, while
every graph of diameter $d$ is $\mathcal{H}$-connected for
$\mathcal{H}$ the collection of all paths of length at most $d+2$.
So $\mathcal{H}$-connectedness naturally extends basic notions of
connectivity.

The definition of $\mathcal{H}$-connectedness was introduced by
Duke, Erd\H{o}s, and R\"odl, who initiated the study of this notion
in a series of four papers \cite{DuEr,DuErRo2,DuErRo3,DuErRo4}. A
graph is {\it $C_{2k}$-connected} if it is $\mathcal{H}$-connected
where $\mathcal{H}$ consists of all even-length cycles of length at
most $2k$. The question studied by Duke, Erd\H{o}s, and R\"odl was
to determine the maximum number of edges in a $C_{2k}$-connected
subgraph that one can find in every graph with $n$ vertices and $m$
edges as a function of $k$, $n$, and $m$. The following problem was
considered to be one of the main open problems in this area. It was
first posed by Duke, Erd\H{o}s, and R\"odl \cite{DuErRo2} in 1984,
and discussed in the two subsequent papers \cite{DuErRo3,DuErRo4}.
It also appears in the book {\it Erd\H{o}s on Graphs} by Chung and
Graham \cite{ChGr}.

\begin{problem}
\label{problem}
Is it true that there are constants $c, \beta_0>0$ such that for all
$0\leq \beta \leq \beta_0$ the following holds. Every graph $G$ with $n$ vertices and $n^{2-\beta}$
edges contains a subgraph $G'$ with $cn^{2-2\beta}$
edges such that every two edges of $G'$ lie together on
a cycle of length at most eight?
\end{problem}

It is easy to see that such a result would be best possible up to a multiplicative constant.
Indeed, the bound on the number of edges in $G'$ is tight when $G$ is a collection of
$n^{\beta}$ disjoint complete graphs of size roughly
$n^{1-\beta}$. In \cite{DuErRo2}, Duke, Erd\H{o}s, and R\"odl obtained a weaker result which proves that the
assertion of Problem \ref{problem} is correct if one allows
the cycle length to be at most twelve instead of at most eight. They also showed how to find
in $G$ a $C_6$-connected (and hence also $C_8$-connected) subgraph $G'$ with at least
$cn^{2-3\beta}$ edges.

The analogue of Problem \ref{problem} for graphs of constant density was
solved in \cite{DuErRo3}. In that paper, the authors proved that for each fixed $d>0$, every graph $G$ with $n$
vertices and at least $dn^2$ edges has a subgraph $G'$ on
$(1+o(1))d^2n^2$ edges such that every pair of edges of $G'$ lie together on a
cycle of length at most eight. Unfortunately, the proof of
Duke, Erd\H{o}s, and R\"odl uses Szemer\'edi's regularity lemma and consequently gives nothing
when $d$ tends to zero.

Duke, Erd\H{o}s, and R\"odl \cite{DuErRo2} also asked
whether Problem \ref{problem} holds in the stronger form, where the
subgraph $G'$ has the additional property that edges sharing a vertex lie
together on a cycle of length at most 6. Motivated by this question, we
call a graph {\it strongly $C_{2k}$-connected} if it is
$C_{2k}$-connected and every pair of edges sharing a vertex lie
together on a cycle of length at most $2k-2$.  In this note, we settle Problem
\ref{problem} in its strengthened form for all $\beta < 1/5$.

\begin{theorem}\label{main}
For $0<\beta<1/5$ and sufficiently large $n$, every graph $G$ on $n$
vertices and at least $n^{2-\beta}$ edges has a strongly
$C_8$-connected subgraph $G'$ with at least
$\frac{1}{64}n^{2-2\beta}$ edges.
\end{theorem}

Our proof combines combinatorial ideas together with a probabilistic
argument which may be called {\em dependent random choice}. Early
versions of this technique were developed in the papers
\cite{Go,KR,Su1}. Later, variants were discovered and applied to a
large variety of extremal problems (see, e.g.,
\cite{KS,AKS,Su2,SuSzVu,DER}). In the concluding remarks, we show
how the same proof can be used to obtain a variant of the main graph
theoretic lemma which is used in the proof of the celebrated
Balog-Szemer\'edi-Gowers theorem. Hence, we wonder if our result
might have new applications in Additive Combinatorics.

\section{Proof of Theorem \ref{main}}

Let $\beta<1/5$, $k=n^{\beta}$ and
let $G$ be a graph with $n$ vertices and at least $n^2/k$ edges.
Since $\beta<1/5$ and $n$ is sufficiently large,
we may assume that $n> 2^{20}k^5$. Delete vertices of minimum
degree one by one until the remaining induced subgraph $G_1$ of $G$ has
minimum degree at least $\frac{n}{2k}$. Since the number of vertices
deleted in this process is at most $n$, we have that the number of remaining edges in $G_1$ is at least
$$e(G_1) \geq e(G)-n \cdot \frac{n}{2k} \geq \frac{n^2}{k}-\frac{n^2}{2k}=\frac{n^2}{2k}.$$
Let $H$ be the maximum bipartite subgraph of $G_1$, and let $A$ and
$B$ denote the vertex classes of $H$. Without loss of generality we
can assume that $|B|\leq |A|$. For a vertex $x \in H$ denote by
$d_H(x)$ its degree, i.e., the number of vertices adjacent to $x$ in
$H$. By maximality of $H$, the degree of every vertex in $H$ is at
least half of its degree in $G_1$ and the number of edges in $H$ is
at least half of the number of edges in $G_1$. Hence the minimum
degree in $H$ is at least $\frac{n}{4k}$ and $H$ has at least
$\frac{n^2}{4k}$ edges. For two vertices $x_1, x_2 \in H$ define the
{\it common neighborhood} $N_H(x_1,x_2)$ to be the set of vertices
of $H$ adjacent to both $x_1$ and $x_2$ and the {\it codegree}
$d_H(x_1,x_2)$ to be the size $|N_H(x_1,x_2)|$. We will later use
the following simple fact.

\begin{lemma}
\label{l1}
If every pair of vertices in a subset $X \subset A$
have codegree in $H$ at most $\frac{n}{32k^2}$, then $|X| <8k$.
\end{lemma}
\begin{proof}
Suppose for contradiction that there is a subset
$X=\{x_1,\ldots,x_{8k}\}$ such that every pair of vertices
in it have codegree in $H$ at most $\frac{n}{32k^2}$. By the
Bonferroni inequality (inclusion-exclusion principle), the number of vertices of $B$
adjacent to at least one of $x_1,\ldots,x_{8k}$ is at least
$$\sum_{1 \leq i \leq 8k} d_H(x_i)-\sum_{1 \leq i<j \leq 8k} d_H(x_i,x_j)
\geq 8k\frac{n}{4k}-{8k \choose 2}\frac{n}{32k^2}>n.$$
Therefore, the size of $B$ is larger than the total number of vertices $n$. This contradiction completes the proof.
\end{proof}

Define an auxiliary graph $\Gamma$ on $A$ where two vertices in $\Gamma$ are
adjacent if their codegree in $H$ is at least $\frac{n}{32k^2}$.
Then the previous lemma simply states that $\Gamma$ has no independent set of size
$8k$. Let $\Gamma'$ be an induced subgraph of $\Gamma$ on $v \geq 16k$ vertices.
If the number of edges in $\Gamma'$ is at most $\frac{v^2}{32k}$, then its average
degree is at most $\frac{v}{16k}$. Therefore, by
Tur\'an's theorem \cite{Tu}, it has an independent set of size at least
$v/(\frac{v}{16k}+1) \geq 8k$, which  contradicts Lemma \ref{l1}.
Thus we have the following claim.

\begin{lemma}  Every induced subgraph
of $\Gamma$ on $v \geq 16k$ vertices has more than $\frac{v^2}{32k}$
edges.
\end{lemma}

In particular, in any induced subgraph $\Gamma_1$ of $\Gamma$, there are at
most $\frac{n}{2^{12}k^4}$ vertices of degree at most
$\frac{n}{2^{16}k^5}$. Otherwise, the subgraph $\Gamma' \subset \Gamma_1$
induced by the vertices of degree at most $\frac{n}{2^{16}k^5}$
has $v \geq \frac{n}{2^{12}k^4} \geq 16k$ vertices and has
at most
$$\frac{1}{2}v \cdot \frac{n}{2^{16}k^5} =
\frac{1}{32k} v \cdot \frac{n}{2^{12}k^4} \leq \frac{v^2}{32k}$$
edges, contradicting the previous lemma.

We say that a vertex $w \in A$ is {\it bad with respect to a pair
$\{u,v\}$} of vertices of $B$ if $w \in N_H(u,v)$ and $w$ has degree
at most $\frac{n}{2^{16}k^5}$ in the induced subgraph
$\Gamma[N_H(u,v)]$ of the auxiliary graph $\Gamma$.

\begin{lemma}
\label{l2}
Let $u,v$ be two vertices in $B$. Pick a vertex $w$ in $A$ uniformly
at random. Let $\cal E$ be the event that $w$ is bad with respect to the
pair $\{u,v\}$. The probability of event $\cal E$ is at most
$\frac{n}{2^{12}k^4|A|}$.
\end{lemma}

\noindent {\bf Proof:}\hspace{1mm} Let $t$ denote the cardinality of
$N_H(u,v)$. The probability that $w \in N_H(u,v)$ is given by
$|N_H(u,v)|/|A| = t/|A|$. Since, by discussion in the previous
paragraph at most $\frac{n}{2^{12}k^4}$ vertices in $N_H(u,v)$ have
degree at most $\frac{n}{2^{16}k^5}$ in the induced subgraph
$\Gamma[N_H(u,v)]$ of $\Gamma$, then the probability that a vertex
picked uniformly at random from $N_H(u,v)$ has degree at most
$\frac{n}{2^{16}k^5}$ in $\Gamma[N_H(u,v)]$ is at most
$\frac{n}{2^{12}k^4}\frac{1}{t}$. Hence, the probability of the
event $\cal E$ satisfies
$$\hspace{1.7cm}
\mathbb{P}[{\cal E}] =\mathbb{P}[w \in N_H(u,v)] \cdot
\mathbb{P}[w~\mbox{is bad}~|~w \in N_H(u,v)]\leq \frac{t}{|A|}\cdot\frac{n}{2^{12}k^4}\frac{1}{t}=\frac{n}{2^{12}k^4|A|}.
 \hspace{1.7cm} \Box$$

Pick a vertex $w \in A$ uniformly at random. Let $Y$ be the random
variable counting the number of pairs $\{u,v\}$ in $B$
such that $w$ is bad with respect to $\{u,v\}$. Since there are
${|B| \choose 2}$ pairs of elements of $B$ and
$|B| \leq |A| \leq n$, then
by Lemma \ref{l2}, we have
$$\mathbb{E}[Y] \leq {|B| \choose
2}\frac{n}{2^{12}k^4|A|} < \frac{n^2}{2^{13}k^4}.$$ Hence there
exist a choice of $w$ such that the number of pairs $\{u,v\}$ in $B$
for which $w$ is bad is less than $\frac{n^2}{2^{13}k^4}$. Pick such
a $w$ and delete all vertices from $A$ that have fewer than
$\frac{n}{32k^2}$ neighbors in $N_H(w)$. That is, delete those
vertices in $A$ that are not adjacent to $w$ in auxiliary graph
$\Gamma$. Let $A'$ be the remaining subset of $A$.

Delete one by one vertices $v$ from $N_H(w)$ for which there are at least
$\frac{n}{2^7k^2}$  vertices $u$ in the remaining set such that
$w$ is bad for $\{u,v\}$. Since $w$ is bad only for at most $\frac{n^2}{2^{13}k^4}$ pairs, it is easy to see that
we deleted at most
\begin{equation}
\label{eq1}
\frac{n^2/(2^{13}k^4)}{n/(2^7k^2)}=\frac{n}{2^6k^2}
\end{equation}
vertices. Denote the remaining subset of $N_H(w)$ by $B'$.  Note
that $|B'| \geq |N_H(w)|-\frac{n}{2^6k^2}=\frac{n}{4k}-\frac{n}{2^6k^2}\geq \frac{n}{5k}$.
By definition of $B'$, we have that for every
$v \in B'$, there are fewer than
$\frac{n}{2^7k^2}$ vertices $u \in B$ such that $w$ is bad for pair $\{u,v\}$.
Let $G'$ be the bipartite subgraph of $H$ induced by $A' \cup
B'$. We will show that this graph satisfies the assertion of Theorem \ref{main}.
The next lemma summarizes several important properties of $G'$.

\begin{lemma}
\label{l3}
{\bf(i)}\, The degree in $G'$ of every vertex in $A'$ is at least $\frac{n}{2^6k^2}$.\\
{\bf (ii)}\, For every vertex $v \in B'$ there are fewer than $\frac{n}{2^7k^2}$ vertices $u \in B'$ such that
$\{v,u\}$ have less than $\frac{n}{2^{16}k^5}$ common neighbors in $A'$.\\
{\bf (iii)}\, The number of edges in $G'$ is at least
$\frac{n^2}{2^6k^2}$.
\end{lemma}

\begin{proof} {\bf (i)}\, Recall that to obtain $A'$ we removed from $A$ all vertices of small degree
in $N_H(w)$.
Thus the vertices in $A'$ all have degree at least
$\frac{n}{32k^2}$ in $N_H(w)$. Also by (\ref{eq1}), we deleted at most
$\frac{n}{2^6k^2}$ vertices from $N_H(w)$ to obtain $B'$. Therefore, all vertices from $A'$ still have
at least $\frac{n}{2^6k^2}$ remaining neighbors in $B'$.

{\bf (ii)}\, Let $\{v,u\}$ be a pair of vertices in $B'$ for which
$w$ is good. By definition, this means that there are at least
$\frac{n}{2^{16}k^5}$ vertices $z$ in $A$ such that $z$ is a common
neighbor of $\{v,u\}$ and the codegree of $z$ and $w$ is at least
$\frac{n}{32k^2}$. All these vertices $z$ have high degree in
$N_H(w)$ and were not deleted when we constructed $A'$. This implies
that there are at least $\frac{n}{2^{16}k^5}$  common neighbors of
pair $\{v,u\}$ in $A'$. To conclude the proof of this part note that
by our construction for every vertex $v \in B'$ there are less than
$\frac{n}{2^7k^2}$ vertices $u \in B'$ such that $w$ is bad for
$\{v,u\}$.

{\bf (iii)}\, Since the minimum degree in $H$ is at least $\frac{n}{4k}$,
we have that $|N_H(w)| \geq \frac{n}{4k}$ and the number of edges between $N_H(w)$ and $A$
is at least $\frac{n}{4k}|N_H(w)| \geq \frac{n^2}{16k^2}$.
Since the vertices we deleted from $A$ all have degree at most
$\frac{n}{32k^2}$ in $N_H(w)$, the total number of remaining edges
between $A'$ and $N_H(w)$ is at least
$$ \frac{n^2}{16k^2}-\frac{n}{32k^2}|A| \geq \frac{n^2}{16k^2}-\frac{n^2}{32k^2}=\frac{n^2}{32k^2}.$$
By (\ref{eq1}), the number of edges between $A'$ and $N_B(w) \setminus B'$ is at most
$$|A||N_H(w) \setminus B'| \leq n \cdot \frac{n}{2^6k^2} \leq \frac{n^2}{2^6k^2}.$$
Hence, the number of edges between $A'$ and $B'$, which is the number of
edges of $G'$, is at least $\frac{n^2}{32k^2}-\frac{n^2}{2^6k^2}=\frac{n^2}{2^6k^2}$.
\end{proof}

Having finished all the necessary preparation we are now ready to complete the proof of
Theorem \ref{main}. Recall that $n>2^{20}k^5$ and let $(a,b),(a',b') \in A' \times B'$ be two
edges of $G'$.

{\bf Case 1:} $(a,b)$ and $(a',b')$ do not share a vertex. By
properties (i) and (ii) of Lemma \ref{l3}, there are at least
$d_{G'}(a)-\frac{n}{2^7k^2}\geq \frac{n}{2^6k^2}-\frac{n}{2^7k^2}=\frac{n}{2^7k^2}$ neighbors $b_1$ of $a$ such that
$b'$ and $b_1$ have  at least $\frac{n}{2^{16}k^5}>4$ common neighbors in $A'$.
Fix any such $b_1\not =b$ and let $a_1$ be a common neighbor of $\{b',b_1\}$ which is different from
$a, a'$. Similarly, we can pick a neighbor $b_2$ of $a'$ different from $b, b', b_1$ such that
$b$ and $b_2$ have at least $\frac{n}{2^{16}k^5}>4$ common neighbors in $A'$.
Let  $a_2$ be a common neighbor of $\{b,b_2\}$ which is distinct from
$a, a', a_1$. Then $a,b_1,a_1,b',a',b_2, a_2,b,a$ form an 8-cycle which contains edges
$(a,b), (a',b')$.

{\bf Case 2:} $a=a'$. Let $a_1$ be a neighbor of $b$ with $a_1 \not
=a'$. (Note that by property (ii) of the previous lemma the degree
of $b$ is at least $\frac{n}{2^{16}k^5}>4$). Then, as in the
previous case, we have that there is a neighbor $b_1$ of $a_1$
different from $b, b'$ such that $b'$ and $b_1$ have  at least
$\frac{n}{2^{16}k^5}>4$ common neighbors in $A'$. Let  $a_2$ be a
common neighbor of $\{b',b_1\}$ which is distinct from $a, a_1$.
Then $a,b,a_1,b_1,a_2,b', a$ form a 6-cycle which contains edges
$(a,b), (a,b')$.

{\bf Case 3:} $b=b'$. Let $b_1$ be a neighbor of $a$ with $b_1 \not
=b$. Then again, as in case 1, there is neighbor $b_2$ of $a'$ different from $b, b_1$ such that
$b_1$ and $b_2$ have  at least $\frac{n}{2^{16}k^5}>4$ common neighbors in $A'$.
Let  $a_2$ be a common neighbor of $\{b_1,b_2\}$ which is distinct from
$a, a'$. Then $a,b,a',b_2,a_2,b_1, a$ form a 6-cycle which contains edges
$(a,b), (a',b)$. \hfill $\Box$

\section{Concluding Remarks}
\begin{itemize}

\item
We suspect that the approach which was used to settle Problem \ref{problem} with some changes might work also
for values of $\beta$ larger than $1/5$. On the other hand, since for our proof it is crucial
to have vertices with large codegree, it surely fails if $\beta \geq 1/2$. It
would be very interesting to determine all values of $\beta$ for which Problem \ref{problem}
have a positive answer. For every $\beta$ that is sufficiently close to 1 there are graphs with
$n^{2-\beta}$ edges and no 8-cycle (see, e.g., \cite{B}). Clearly, for such $\beta$
the answer to this problem is negative.

\item Duke, Erd\H{o}s, and R\"odl \cite{DuErRo2} showed that for
$0<\beta< 1/2$ and a graph $G$ on $n$ vertices and at least
$n^{2-\beta}$ edges, there is a $C_6$-connected subgraph $G'$ on at
least $cn^{2-3\beta}$ edges, and this result is tight up to the
multiplicative constant $c$. However, it is still open whether this
result can be strengthened to show that every graph $G$ with $n$
vertices and $n^{2-\beta}$ edges has a strongly $C_6$-connected
subgraph $G'$ with at least $cn^{2-3\beta}$ edges. Duke, Erd\H{o}s,
and R\"odl \cite{DuErRo2} proved that such a graph $G$ will have a
strongly $C_6$-connected subgraph $G'$ with at least $cn^{2-5\beta}$
edges.

\item The Balog-Szemer\'edi-Gowers theorem is a very useful tool in Additive
Combinatorics. For example, it is an important ingredient in Gowers' proof
\cite{Go} of Szemer\'edi's theorem on arithmetic progressions in
dense sets. For detailed discussion and more applications of this theorem see, e.g., the
books by Nathanson \cite{Na} and by Tao and Vu \cite{TaVu}.

Let $A$ and $B$ be sets of integers. The {\it sumset} $A+B$ is
defined to be the collection of sums $a+b$ with $a \in A$, $b \in
B$. For a bipartite graph $G=(A,B;E)$, the {\it partial sumset}
$A+_G B$ is defined to be the collection of sums $a+b$ with $(a,b)
\in E(G)$. The Balog-Szemer\'edi theorem \cite{BaSz} states that for $A$ and
$B$ sets of $n$ integers with $|E(G)| \geq n^2/k$ and $|A+_G B| \leq
cn$ for some $k$ and $c$, there are $A' \subset A$ and $B' \subset
B$ such that $|A'|,|B'| \geq n/K$ and $|A'+B'| \leq Cn$, where $K$ and
$C$ only depend on $k$ and $c$. The original proof of this theorem gave
a poor bound on $K$ and $C$ in terms of $k$ and $c$. Gowers \cite{Go}
discovered a new proof in which $K$ and $C$ can be taken to be
polynomials in $k$ and $c$.

One can deduce the Balog-Szemer\'edi-Gowers theorem rather quickly
from a graph-theoretic lemma proved by Sudakov, Szemer\'edi,
and Vu \cite{SuSzVu}, which essentially says that for every dense
bipartite graph $G=(A,B;E)$ with $|A|=|B|$, there are linear-sized
subsets $A' \subset A$ and $B' \subset B$ such that for every pair
$(a,b) \in A' \times B'$, there are a quadratic number of paths in $G$
of length three between $a$ and $b$. The following theorem
strengthens this graph-theoretic lemma, showing that the paths of
length three can be taken to lie entirely within subgraph of $G$ induced by $A' \cup B'$.

\begin{theorem}\label{graphtheorylemma}
For every bipartite graph $G=(A,B;E)$ with $n \geq 2^{18}k^5$
vertices and $|E| \geq n^2/k$ edges, there are subsets $A' \subset A$
and $B' \subset B$ such that the subgraph $G'$ of $G$ induced by $A'
\cup B'$ has at least $\frac{n^2}{2^6k^2}$ edges and for every $a
\in A'$ and $b \in B'$, there are at least $\frac{n^2}{2^{24}k^7}$ paths
between $a$ and $b$ in $G'$ of length three.
\end{theorem}
\begin{proof}
The proof follows easily from Lemma \ref{l3}.
By properties (i) and (ii) of this lemma, there are at least
$d_{G'}(a)-\frac{n}{2^7k^2}\geq \frac{n}{2^6k^2}-\frac{n}{2^7k^2}=\frac{n}{2^7k^2}$ neighbors $b_1$ of $a$ such that
pair $\{b,b_1\}$ have  at least $\frac{n}{2^{16}k^5}$ common neighbor in $A'$. For any such $b_1 \not =b$ and any common neighbor $a_1\not =a$
we have a path of length
three $a, b_1, a_1, b$. The number of such paths is clearly
at least $\left(\frac{n}{2^7k^2}-1\right) \left(\frac{n}{2^{16}k^5}-1\right) \geq \frac{n^2}{2^{24}k^7}$.
\end{proof}

We wonder if this theorem might have new applications in Additive Combinatorics.
\end{itemize}

\noindent {\bf Acknowledgment.}\, We would like to thank Daniel
Martin for pointing out an error in an earlier version of this
paper.

\end{document}